\documentclass[a4paper,12pt]{article}
\usepackage[toc,page]{appendix}
\usepackage[numbers]{natbib}
\usepackage{rotating}
\usepackage{mathrsfs}
\usepackage{amsfonts}
\usepackage{amssymb}
\usepackage{caption}
\usepackage{amsmath}
\usepackage{amsthm}
\usepackage{url}
\usepackage{accents}
\usepackage{graphicx}
\usepackage[math]{}
\usepackage[toc,page]{appendix}
\usepackage[]{graphicx}
\usepackage{caption}
\usepackage{subcaption}
\usepackage{authblk}
\title{On the algebraic properties of exponentially stable integrable hamiltonian systems}
\author{Santiago Barbieri\footnote{D\'epartement de Math\'ematiques d'Orsay - Universit\'e Paris Saclay, B\^at. 307, Rue Michel Magat, 91400 Orsay, France\\ 
		Dipartimento di Matematica e Fisica - Universita degli Studi Roma Tre, Largo San Leonardo Murialdo, 1, Palazzina C, 00146 Roma, Italy
		\\ This work has partially been developed under the auspices of the European Research Council in the framework of the H2020-ERC Starting Grant 2015 project 677793: {\itshape Stable and Chaotic Motions in the Planetary Problem.}} }
\date{}
\begin{document}
  \maketitle
  \section*{Abstract}
  Steepness is a geometric property which, together with complex-analyticity, is needed in order to insure stability of a near-integrable hamiltonian system over exponentially long times. Following a strategy developed by Nekhoroshev, we construct sufficient algebraic conditions for steepness for a given function that involve algebraic equations on its derivatives up to order five. The underlying analysis suggests some interesting considerations on the genericity of steepness and represents a first step towards the construction of sufficient conditions for steepness involving the derivatives of the studied function up to an arbitrary order. 
\section{Introduction}
Hamiltonian formalism is the natural setting appearing in the study of many physical systems. In the simplest case, we consider the motion of a point on a Riemannian manifold $\mathcal{M}$, called configuration manifold, governed by Newton's second law ($\ddot{q}=-\nabla U(q)$ for a potential function $U$ in the euclidean case, with $q$ a system of local coordinates for $\mathcal{M}$). This system can be transformed by duality thanks to Legendre's transformation and reads 
$$
\dot{p}=-\partial_q H(p,q)\quad , \qquad \dot{q}=\partial_p H(p,q) \quad,
$$
where $H(p,q)$ is a real differentiable function on the cotangent bundle $T^* \mathcal{M}$, classically called hamiltonian, and $p$ is the coordinate conjugated to $q$. Systems integrable by quadrature are an important class of hamiltonian systems. By the classical Liouville-Arnol'd Theorem, under general topological assumptions, an integrable system depending on $2n$ variables ($n$ degrees of freedom) can be conjugated to a hamiltonian system on the cotangent bundle of the $n$-dimensional torus $\mathbb{T}^n$, whose equations of motion take the form 
$$
\dot{I}=-\partial_\vartheta h(I)=0\quad , \qquad \dot{\vartheta}=\partial_I h(I) \quad,
$$
where $(I,\vartheta)\in \mathbb{R}^n\times \mathbb{T}^n $ are called action-angle coordinates. Therefore, the phase space for an integrable system is foliated by invariant tori carrying the linear motions of the angular variables (called quasi-periodic motions). Integrable systems are exceptional, but many important physical problems are governed by hamiltonian systems which are close to integrable. 
  Namely, the dynamics of a near-integrable hamiltonian system is described by a hamiltonian function whose form in action-angle coordinates $(I,\vartheta)\in \mathbb{R}^n\times \mathbb{T}^n$ reads
$$
H(I,\vartheta):=h(I)+\varepsilon f(I,\vartheta)\ ,
$$
where $\varepsilon$ is a small parameter. The structure of the phase space for this kind of systems can be inferred with the help of Kolmogorov-Arnol'd-Moser (KAM) theory. Namely, under a general non-degeneracy condition for $h$, a Cantor set of positive measure of invariant tori carrying quasi-periodic motions for the integrable flow persists under a suitably small perturbation (see e.g. ref. \cite{Arnold_2010}, \cite{Chierchia_2008}).
\newline
For systems with three or more degrees of freedom, KAM theory yields little information about trajectories lying in the complementary of such Cantor set, where instabilities can occur (see e.g. ref. \cite{Arnold_1964}). However, in a series of articles published during the seventies (see ref. \cite{Nekhoroshev_1977}, \cite{Nekhoroshev_1979}) Nekhoroshev proved an effective result of stability for an open set of initial conditions holding over a time which is exponentially long in the inverse of the size $\varepsilon$ of the perturbation, provided that the hamiltonian is analytic and that its integrable part satisfies a generic transversality property known as {\it steepness}.
\newline
From a more technical point of view, steepness is defined as follows:
\newtheorem{steepness}{Definition}
\begin{steepness}
Let $\mathcal{A}$ be an open set of $\mathbb{R}^n$ and $h:\mathcal{A}\longrightarrow \mathbb{R}$ a smooth function. $h$ is {\itshape steep} at $I:=(I_1,...,I_n)\in \mathcal{A}$ if 
$
\nabla h (I)\neq 0
$
and if, for any $m=1,...,n-1$, there exist constants $C_m>0$, $\delta_m>0$ and $\alpha_m>1$ such that, for all $m$-dimensional affine subspace $\Lambda^I_m$ orthogonal to $\nabla h(I)$, the gradient of the restriction of $h$ to $\Lambda^I_m$, which we denote with $\nabla (h|_{\Lambda^I_m})$, satisfies
\begin{equation}
\max_{0\leq \eta\leq \xi}\left(\min_{I^{'}\in \Lambda^I_m,\ ||I-I^{'}||=\eta}||\nabla (h|_{\Lambda^I_m})(I^{'})||\right)>C_m \xi^{\alpha_m}\ ,\ \ \forall \xi \in (0,\delta_m]\ .
\end{equation}
\end{steepness}
The constants $C_m$ and $\delta_m$ are called the steepness coefficients of $h$, whereas the $\alpha_m$ are its steepness indices. In particular, in the analytic case, a function is steep if and only if, on any affine hyperplane $\Lambda^I_m$, there exists no curve $\gamma$  with one endpoint in $I$ such that the restriction $\nabla (h|_{\Lambda^I_m})$ identically vanishes on $\gamma$, as is showed in ref. \cite{Niederman_2006}. From a heuristic point of view, for any value $m\in\{1,...,n-1\}$ the gradient $\nabla h$ must "bend" towards $\Lambda_m^I$ when "travelling" along the curve $\gamma\in \Lambda_m^I$, so that critical points for the restriction of $h$ to $\Lambda^I_m$ must not accumulate (see ref. \cite{Niederman_2006}). Finally, $h$ is said to be steep in a given domain if it is steep at each point of such set with uniform indices and coefficients.
\newline
With such notion, Nekhoroshev's effective result of stability reads
\newtheorem{nek}{Theorem (Nekhoroshev, 1977)}
\begin{nek}
Consider a near-integrable system with hamiltonian $
H(I,\vartheta):=h(I)+\varepsilon f(I,\vartheta)
$ analytic in some complex neighborhood of $B_r\times \mathbb{T}^n$, where $B_r$ is the open ball of radius $r$ in $\mathbb{R}^n$, and suppose $h$ steep. Then there exist positive constants $a,b,\varepsilon_0, C_1,C_2$ such that, for any $\varepsilon\in [0,\varepsilon_0)$ and for any initial condition not too close from the boundary, one has
$
|I(t)-I(0)|\leq C_2\varepsilon^a
$
for any time $t$ satisfying
$
|t|\leq C_1\exp\left(\varepsilon^{-b}\right)\ .
$
\end{nek}
Such result also holds under the weaker regularity assumption that the hamiltonian is in the Gevrey class (see ref. \cite{Marco_Sauzin_2002}) and by requiring steepness to be verified only on those subspaces which are spanned by integer vectors satisfying suitable arithmetic conditions (see refs. \cite{Guzzo_Morbidelli_1997}, \cite{Niederman_2007}). However, one cannot get completely rid of the steepness hypothesis since examples of instability over times of order $1/\varepsilon$ may be constructed in case such property is not satisfied on a subspace spanned by integer vectors (see ref. \cite{Niederman_2006}, \cite{Bounemoura_Kaloshin_2014}).
Therefore, a crucial step in order to establish stability over exponentially long times for a near-integrable hamiltonian system consists in building a suitable steep integrable approximation. This aspect is important when trying to apply Nekhoroshev's estimates to concrete examples, as it is shown for example in refs. \cite{Bounemoura_Fayad_Niederman_2017} and \cite{Pinzari_2013}. As we shall see, the steepness property is generic, both in measure and topological sense. However, since its definition is not constructive, it is difficult to directly establish wether a given function is steep or not. Fortunately, Nekhoroshev provided in \cite{Nekhoroshev_1979} a scheme which, in principle, allows to deduce explicit sufficient algebraic conditions for steepness involving the derivatives of the studied function up to an arbitrary order. In particular, let us define the $r$-jet $P_I(h,r,n)$ of a smooth function $h$ of $n$ variables at $I$ as the vector containing all the coefficients of the Taylor polynomial of $h$ at $I$ up to order $r$, with the exception of the constant term, namely
$$
P_I(h,r,n):=\left\{\frac{1}{\mu !}\frac{\partial^\mu h}{\partial I^\mu},\ 1 \leq |\mu|\leq r\right\}\ ,
$$
where $\mu:=(\mu_1,...\mu_n)$ is a multi-index of naturals and $|\mu|=\sum_{i=1}^n\mu_i$.
\newline
With this definition, one can pass to the quotient in the set of smooth functions and consider a representative of the class of smooth functions of $n$ variables having the same $r$-jet at $I$. We also denote with $\mathcal{P}_I(r,n)$ the polynomial space of the $r$-jets of smooth functions of $n$ variables calculated at $I$.  Nekhoroshev showed that, for any $r\geq 2,$ one can construct a semi-algebraic set whose closure contains the $r$-jets of all non-steep functions with non-zero gradient at $I$.
Namely, we have the following 
\newtheorem{nek2}[nek]{Theorem (Nekhoroshev, 1979)}
\begin{nek2}\label{nek2}
For any $n\geq 2$ and $r\geq 2$, there exists a semi-algebraic set $\sigma_n^r(I)\subset \mathcal{P}_I(r,n)$, whose closure is denoted with $\Sigma_n^r(I)$, such that any given function $h$ satisfying:
\begin{enumerate}
\item $h\in C^{2r-1}$ in a neighborhood of $I$\ ,
\item  $\nabla h(I)\neq 0$\ ,
\item $P_I(h,r,n)\in \mathcal{P}_I(r,n)\backslash \Sigma_n^r(I)$\ ,
\end{enumerate}
is steep in some neighborhood of $I$. 
\newline
Moreover, for any $m=1,...,n-1$, one has 
\begin{equation}\label{codimension}
\text{codim }\Sigma_n^r\geq 
\begin{cases}
\max\left\{0,r-1-\displaystyle\frac{n(n-2)}{4}\right\},\ \text{if $n$ is even}\\
\max\left\{0,r-1-\displaystyle\frac{(n-1)^2}{4}\right\},\ \text{if $n$ is odd}\\
\end{cases}
\end{equation}
and the steepness indices $\alpha_m$ of $h$ are superiorly bounded by 
\begin{equation}\label{boundam}
\bar{\alpha}_m:=
\begin{cases}
\max\left\{1,2r-3-\displaystyle\frac{n(n-2)}{2}+2m(n-m-1)\right\},\ \text{if $n$ is even}\\
\max\left\{1,2r-3-\displaystyle\frac{(n-1)^2}{2}+2m(n-m-1)\right\},\ \text{if $n$ is odd}
\end{cases}\ .
\end{equation}
\end{nek2}
As Nekhoroshev points out in his discussion, such result implies a stratification in the space of jets: the strata $\Sigma_n^r(I)$, with $r\geq 2$, are semialgebraic sets of increasing codimension. Hence, as expression \eqref{codimension} shows, for fixed $n$ and sufficiently high $r$, the steepness property is generic in $\mathcal P_I(r,n)$. Moreover, for fixed values of $r$ and $n$, in addition to non-steep functions, also all steep functions with steepness indices greater than $\bar{\alpha}_m$ are contained in the stratus $\Sigma_n^r(I)$. In other words, for increasing values of $r$, the complementary of $\Sigma_n^r(I)$ contains more and more jets of steep functions and the steepness indices of such functions are superiorly bounded by a quantity $\bar{\alpha}_m$ which increases linearly with $r$. A way to obtain sufficient conditions for steepness in the space of jets at a fixed order $r$ consists therefore in knowing the explicit form of the stratus $\Sigma_n^r(I)$ or the form of some closed set containing it: a function whose $r$-jet lies outside such set is steep. The sets $\Sigma^2_n(I)$, for any $n\geq 2$, have been explicitly described by Nekhoroshev in references \cite{Nekhoroshev_1977} and \cite{Nekhoroshev_1979}. Before stating Nekhoroshev's results, we denote with
$$
h_I^k[v^1,...,v^k]=\sum_{i_1,...,i_k=1}^n \frac{\partial^k h}{\partial I_{i_1}...\partial I_{i_k}}(I)v^1_{i_1}...v^k_{i_k}
$$
the $k$-th order multilinear form corresponding to the $k$-th coefficient of the Taylor expansion of a function $h$ which is $k$-times continuously differentiable in a neighborhood of $I$. We also give the following
\newtheorem{rjetcondition}[steepness]{Definition}
\begin{rjetcondition}\label{rjetcondition}
	For $r\in\mathbb{N}$, $r\ge 2$, a function $h$ of class $C^r$ in a neighborhood of a point $I$ is said to be $r$-jet non-degenerate if the system
	$$
	h^1[v]=0\ ;\ \ h^2[v,v]=0\ ,\ \ ...\ ,\ \ h^r[v,...,v]=0
	$$
	admits only the trivial solution $v=0$. If this is not the case, $h$ is said to be $r$-jet degenerate.
\end{rjetcondition}
With such setup, Nekhoroshev proved that, in the space of jets of order two, one has $P_I(h,2,n)\in\mathcal{P}_I(2,n)\backslash\Sigma_n^2(I)$ if and only if $h$ is two-jet degenerate. Such condition is equivalent to requiring that $h$ is quasi-convex (i.e. convex on level sets) at $I$ and Theorem \ref{nek2} implies that all quasi-convex functions in $C^{3}$ class around a non-critical point $I$ are steep in a neighborhood of such point. 
In a similar way, Nekhoroshev found a sufficient condition for steepness involving the derivatives of order three: namely, if a function $h\in C^{5}$ around a non-critical point $I$ is three-jet non-degenerate at $I$,
then $h$ is steep in a neighborhood of such point. As we shall see in subsection \ref{discussion2}, such result is more general than the conditions that can be inferred on the jets of order three by simply following the scheme of Theorem \ref{nek2}, since it applies to a wider set of functions. Its proof is not found in Nekhoroshev's works (see refs. \cite{Nekhoroshev_1973} and \cite{Nekhoroshev_1979}) and it has been explicitly written in an analytic way in \cite{Chierchia_Faraggiana_Guzzo_2019} for systems with any number of degrees of freedom. As we shall show in subsection \ref{discussion2}, the fact that the three-jet non degeneracy of a given function depending on $n=2,3,4,5$ coordinates implies its steepness comes out as a straightforward corollary of the algebraic structure of the equations that define the sets $\Sigma_n^r(I)$, for $r=4,5$ and $n=2,3,4,5$. Following such discussion, we conjecture that the algebraic properties of the sets $\Sigma_n^r(I)$, for any values of $r$ and $n$, can be used to prove the steepness of all three-jet non-degenerate functions depending on an arbitrary number of variables; this would constitute an alternative proof to the one used in \cite{Chierchia_Faraggiana_Guzzo_2019}.
\newline
However, the algebraic form of the strata $\Sigma_n^r(I)$, for $r\geq 4$, cannot be expressed so straightforwardly as in the cases $r=2,3$. In \cite{Schirinzi_Guzzo_2013} the authors were able to build closed sets containing the strata $\Sigma_n^4(I)$ for $n=2,3,4$ by exploiting Nekhoroshev's strategy. For $n\geq 5$ and $r=4$, on the other hand, Nekhoroshev's scheme turns out not to be helpful since it yields conditions for steepness which are stronger than three-jet non degeneracy. 
\newline
In this work, we develop the scheme in \cite{Schirinzi_Guzzo_2013} and we build closed sets containing $\Sigma_n^5(I)$ for $n=2,3,4,5$. This allows us to formulate new explicit conditions for steepness involving the five-jet of a given function. Similarly to the case considered in \cite{Schirinzi_Guzzo_2013}, the constraints we find are useful only in the case of systems with $n=2,3,4,5$ degrees of freedom, as we shall discuss in subsection \ref{ngeq6}. Moreover, we slightly modify the construction in \cite{Schirinzi_Guzzo_2013} so to get rid of some hypotheses of non-degeneracy on the hessian matrix of the function whose steepness is being tested. Furthermore, this work can be seen as a first step towards the formulation of sufficient conditions for steepness in the space of jets of arbitrary order. Indeed, a comparison on the equations defining the 'bad' sets $\sigma_n^r(I)$ defined in Theorem \ref{nek2}, for $r=4,5$, suggests hints on the algebraic structure of $\sigma_n^r(I)$ for any value of $r$, which shall be studied in detail in a further work. By formula \eqref{codimension}, this would allow to obtain generic conditions for steepness for functions depending on an arbitrary number of degrees of freedom. Actually, if the explicit expression of the sets $\sigma_r^n(I)$ were known for all $r,n\in \mathbb{N}$, for any fixed value of $n$ one should then simply find the minimal order $r^*$, depending on $n$, for which the codimension of the bad set $\sigma_n^r(I)$ is positive. At that point, steepness would be generic in the space of jets of order $r^*$ and a way to test steepness of a given function would be to see if its $r^*$-jet belongs to the complementary of the closure of $\sigma_n^{r^*}(I)$.

This paper is organized as follows: in section \ref{results} we state our results, whereas in section \ref{examples} we test such conditions on a couple of polynomial examples. Section \ref{nekhoroshev} is dedicated to an overview on Nekhoroshev abstract strategy to construct sets $\sigma_n^r(I)$, section \ref{proofs} contains the proofs of the statements in section \ref{results} and, finally, section \ref{discussion} contains some remarks and a short discussion on the possible developements of this work. 
\section{Results}\label{results}

Below, we state our results separately for each of the possible values of the number of degrees of freedom $n$.
\\
 As a matter of notation, for fixed $n$ and for any collection of $m\in\{1,...,n-1\}$ vectors $v_1,...,v_m$ in $\mathbb{R}^n$, we shall indicate with $rk(v_1,...,v_m)$ the linear rank of the matrix $(v_1,...,v_m)$ generated by such collection. 
\newline 

For $n=2$ we have
\newtheorem{n2}[nek]{Theorem}
\begin{n2}\label{n2}
Let $\mathcal{A}$ an open set of $\mathbb{R}^2$ and $h:\mathcal{A}\longrightarrow \mathbb{R}$ a smooth function. Let $I\in\mathcal{A}$ a point such that $\nabla h(I)\neq 0$.
If $h$ is five-jet non-degenerate at $I$, then $h$ is steep in some neighborhood of $I$. 
\end{n2} 
For $n=3$ we have 
\newtheorem{n3}[nek]{Theorem}
\begin{n3}\label{n3}
Let $\mathcal{A}$ an open set of $\mathbb{R}^3$ and $h:\mathcal{A}\longrightarrow \mathbb{R}$ a smooth function. Let $I\in\mathcal{A}$ a point such that $\nabla h(I)\neq 0$. 
If 
\begin{enumerate}
\item $h$ is five-jet non-degenerate at $I$;
\item for any $v\neq 0$ such that $h$ is three-jet degenerate at $I$,
any vector $u$ solving system
$$
h^1_I[u]=0\ ;\ \
h^2_I[u,v]=0\ ; \ \
h^2_I[u,u]h_I^4[v,v,v,v]=3(h_I^3[v,v,u])^2\\
$$
satisfies $rk(u,v)<2$;
\end{enumerate}
then $h$ is steep in some neighborhood of $I$. 

\end{n3}
For $n=4$ we have
\newtheorem{n4}[nek]{Theorem}
\begin{n4}\label{n4}
Let $\mathcal{A}$ an open set of $\mathbb{R}^4$ and $h:\mathcal{A}\longrightarrow \mathbb{R}$ a smooth function. Let $I\in\mathcal{A}$ a point such that $\nabla h(I)\neq 0$. 
\newline
If 
\begin{enumerate}
\item $h$ is five-jet non-degenerate at $I$;
\item for all $v\neq 0$ such that
$h$ is three-jet degenerate at $I$,
any vector $u$ solving system
\begin{align}\label{n4m2}
\begin{split}
\begin{cases}
h^1_I[u]=0\ ;\ \
h^2_I[u,v]=0\ ;\ \
h^2_I[u,u]h^4_I[v,v,v,v]=3(h^3_I[v,v,u])^2\\
15(h_I^3[v,v,u])^2h_I^3[u,u,v]+h_I^5[v,v,v,v,v](h_I^2[u,u])^2\\
\ \ =10h_I^4[v,v,v,u]h_I^3[u,v,v]h_I^2[u,u]
\end{cases}
\end{split}
\end{align}
satisfies $rk(u,v)<2$;
\item for all $v\neq 0$ such that $h$ is three-jet degenerate at $I$,
any couple of vectors $(u,w)$ solving
\begin{align}\label{n4m3}
\begin{split}
h_I^1[u]=0\ ;\ \
h_I^1[w]=0\ ;\ \
h_I^2[u,v]=0\ ;\ \
h_I^2[w,v]=0
\end{split}
\end{align}
satisfies $rk(u,v,w)<3\ ;$
\end{enumerate}
then $h$ is steep in some neighborhood of $I$. 
\end{n4}
For $n=5$ we have
\newtheorem{n5}[nek]{Theorem}
\begin{n5}\label{n5}
Let $\mathcal{A}$ an open set of $\mathbb{R}^5 $ and $h:\mathcal{A}\longrightarrow \mathbb{R}$ a smooth function. Let $I\in\mathcal{A}$ a point such that $\nabla h(I)\neq 0$. 
\newline
If 
\begin{enumerate}
\item $h$ is four-jet non-degenerate at $I$;
\item for all $v\neq 0$ such that
$h$ is three-jet degenerate at $I$,
any vector $u$ solving
\begin{align}\label{n5m2}
\begin{split}
\begin{cases}
h^1_I[u]=0\ ;\ \
h^2_I[u,v]=0\ ;\ \
h^2_I[u,u]h_I^4[v,v,v,v]=3(h_I^3[v,v,u])^2\\
15(h_I^3[v,v,u])^2h_I^3[u,u,v]+h_I^5[v,v,v,v,v](h^2_I[u,u])^2\\
\ \ =10h^4_I[v,v,v,u]h_I^3[u,v,v]h_I^2[u,u]
\end{cases}
\end{split}
\end{align}
satisfies $rk(u,v)<2$;
\item for all $v\neq 0$ such that
$h$ is three-jet degenerate at $I$,
any couple of vectors $(u,w)$ solving
\begin{align}\label{n5m3}
\begin{split}
\begin{cases}
h^1_I[u]=0\ ;\ \
h^1_I[w]=0\ ;\ \
h^2_I[u,v]=0\ ;\ \
h^2_I[w,v]=0\\
\{h_I^4[v,v,v,v]h_I^2[u,u]-6(h_I^3[u,v,v])^2\}\{h^2_I[w,w]h^2_I[u,u]-(h^2_I[u,w])^2\}\\
+12h^3_I[u,v,v]h^3_I[v,v,w]h^2_I[u,u]h^2_I[u,w]-6(h_I^3[u,v,v])^2(h_I^2[u,w])^2\\
-6(h_I^3[v,v,w])^2(h_I^2[u,u])^2=0
\end{cases}
\end{split}
\end{align}
satisfies $rk(u,v,w)<3\ ;$
\item for all $v\neq 0$ such that 
$h$ is two-jet degenerate at $I$,
any triplet of vectors $(u,w,x)$ solving
\begin{align}\label{n5m4}
\begin{split}
\begin{cases}
h^1_I[u]=0\ ;\ \
h^1_I[w]=0\ ;\ \
h^1_I[x]=0\\
h^2_I[u,v]=0\ ;\ \
h^2_I[w,v]=0\ ;\ \
h^2_I[x,v]=0
\end{cases}
\end{split}
\end{align}
satisfies $rk(u,w,x,v)<4\ ;$
\end{enumerate}
then $h$ is steep in some neighborhood of $I$. 
\end{n5}
\section{Examples}\label{examples}
In this section, we test our results on some polynomial examples.
\newtheorem{ex1}{Example}
\begin{ex1}
The function 
\begin{equation}
h(I)=h(I_1,I_2,I_3,I_4)=\frac{I_2^5}{5}+\frac{I_1^3}{3}-\frac{I_1^2}{2}+\frac{I_1I_2}{2}-\frac{I_3^2}{2}-I_4
\end{equation} 
is steep in a neighborhood of the origin $I=0$.
\end{ex1}
\begin{proof}
We start be remarking that such function is three-jet and four-jet degenerate at the origin on those vectors $v\neq 0$ of the form 
\begin{equation}\label{v}
v:=(v_1,v_2,v_3,v_4)=(0,v_2,0,0)\ ,
\end{equation}  
so that neither Nekhoroshev explicit algebraic conditions for steepness, nor theorems in ref. \cite{Schirinzi_Guzzo_2013} apply.
However, the claim can be proven by applying Theorem \ref{n4}. Indeed, it is easy to see that  $h$ is five-jet non-degenerate at $I=0$. Moreover, system (\ref{n4m2}) reads
\begin{equation}
u_4=0\ ;\ \
u_1v_1+u_3v_3-\frac{1}{2}u_1v_2-\frac{1}{2}u_2v_1=0\ ;\ \
v_2^5(u_1^2+u_3^2-u_1u_2)^2=0
\end{equation}
and, by taking expression (\ref{v}) into account, one has that the only non-null solution is given by vectors of the kind $u=(0,u_2,0,0)$, which satisfy $rk(u,v)<2$.
\newline
Finally, system (\ref{n4m3}) reads
\begin{equation}
\begin{cases}
u_4=w_4=0\\
u_1v_1+u_3v_3-\frac{1}{2}u_1v_2-\frac{1}{2}u_2v_1=0\ ;\ \
w_1v_1+w_3v_3-\frac{1}{2}w_1v_2-\frac{1}{2}w_2v_1=0
\end{cases}
\end{equation}
and, by taking expression (\ref{v}) again into account, the only possible solutions are two families of vectors of the kind $u=(0,u_2,u_3,0)$ and $w=(0,w_2,w_3,0)$, which satisfy $rk(u,v,w)<3$.
Therefore, the hypotheses of Theorem \ref{n4} are fulfilled and the proof is concluded.
\end{proof}
\newtheorem{ex2}[ex1]{Example}
\begin{ex2}
The function 
\begin{equation}
h(I)=h(I_1,I_2,I_3,I_4,I_5)=\frac{I_4^4}{4}+\frac{I_5^4}{4}+\frac{I_3^3}{3}+\frac{I_3I_2^2}{2}-\frac{I_1^2}{2}-\frac{I_3^2}{2}-\frac{I_5^2}{2}+I_3I_4+I_2
\end{equation}
is steep in a neighborhood of the origin $I=0$.
\end{ex2}
\begin{proof}
We start by remarking that such function is two-jet degenerate at the origin on those vectors $z\neq 0$ of the form 
\begin{equation}\label{z}
z:=(z_1,0,z_3,z_4,z_5)
\end{equation}
whose coordinates satisfy  
\begin{equation}\label{zz}
z_1^2+z_3^2+z_5^2-2z_3z_4=0\ ,
\end{equation}
and three-jet degenerate on those vectors $v\neq 0$ of the kind 
\begin{equation}\label{vv}
v:=(0,0,0,v_4,0)\ .
\end{equation}
Therefore, Nekhoroshev's non-degeneracy conditions on the jets of order two and three are helpless in this case. Moreover, since such function has five degrees of freedom, the results in ref. \cite{Schirinzi_Guzzo_2013} cannot be used (they only hold for $n=2,3,4$).
However, the claim can be proven by making use of Theorem \ref{n5}. First, it is easy to see that $h$ is four-jet non-degenerate at the origin. Moreover, by taking expression (\ref{vv}) into account, system (\ref{n5m2}) reads
\begin{equation}
u_2=0\ ;\ \ 
u_3=0\ ;\ \ 
u_1^2+u_5^2=0
\ ,
\end{equation}
which is solved by vectors of the kind $u=(0,0,0,u_4,0)$, that depend linearly on $v$. On the other hand, system (\ref{n5m3}) has the form
\begin{equation}\label{uw}
u_2=0\ ;\ \ 
w_2=0\ ;\ \ 
u_3=0\ ;\ \ 
w_3=0\ ;\ \
(u_1^2+u_5^2)(u_1w_5-u_5w_1)^2=0\ ,
\end{equation}
where the particular form (\ref{vv}) of vector $v$ has been taken into account once again. \newline
There are four possible cases
\begin{enumerate}
\item $u_1^2+u_5^2=0$, which, by system (\ref{uw}), implies $u_1=u_5=0$, so that vector $u$ is of the kind $u=(0,0,0,u_4,0)$, which is parallel to $v$;
\item $u_1=0,u_5\neq 0$, which, by the last equation in (\ref{uw}), implies $w_1=0$, so that $u$ and $w$ have the form $u=(0,0,0,u_4,u_5)$ and $w=(0,0,0,w_4,w_5)$, so that $rk(u,v,w)<3$;
\item $u_1\neq 0, u_5=0$ which is similar to the previous point and yields $u,w$ of the kind $u=(u_1,0,0,u_4,0),w=(w_1,0,0,w_4,0)$, so that $rk(u,v,w)<3$;
\item $u_1\neq 0,u_5\neq 0$ which, by system (\ref{uw}), yields $u_5w_1-u_1w_5=0$. Therefore, one has
$$
\det
\left(
\begin{matrix}
0 & 0 & 0 & v_4 & 0\\
u_1 & 0 & 0 & u_4 & u_5\\
w_1 & 0 & 0 & w_4 & w_5\\
0 & 1 & 0 & 0 & 0\\
0 & 0 & 1 & 0 & 0\\
\end{matrix}
\right)=v_4(u_5w_1-u_1w_5)=0
$$
so that, since 
$$
rk
\left(
\begin{matrix}
0 & 0 & 0 & v_4 & 0\\
0 & 1 & 0 & 0 & 0\\
0 & 0 & 1 & 0 & 0\\
\end{matrix}
\right)=3\ ,
$$
one must have $rk(u,w)<2$ and, consequently, $rk(u,v,w)<3$.
\end{enumerate}
Finally, system (\ref{n5m4}) reads
\begin{equation}
\begin{cases}
u_2=0\ ;\ \ 
w_2=0\ ;\ \ 
x_2=0\\
u_1z_1+u_3z_3+u_5z_5-u_3z_4-u_4z_3=0\\
w_1z_1+w_3z_3+w_5z_5-w_3z_4-w_4z_3=0\\
x_1z_1+x_3z_3+x_5z_5-x_3z_4-x_4z_3=0\\
\end{cases}
\end{equation}
which means that the vectors $u,w,x$ belong to the three-dimensional subspace orthogonal to $Span\{(0,1,0,0,0),(z_1,0,z_3-z_4,-z_3,z_5)\}$. By looking at expressions (\ref{z}) and (\ref{zz}), we see that vector $z$ belongs to the same subspace, so that $rk(u,w,x,z)<4$. The hypotheses of theorem \ref{n5} are therefore fulfilled and the proof is concluded.
\end{proof}
\section{The steepness property in the space of jets}\label{nekhoroshev}
We start by recalling Nekhoroshev's abstract construction (see ref. \cite{Nekhoroshev_1979}) of the sets $\sigma_n^r(I)$ in Theorem \ref{nek2}, with $r,n\geq 2$, whose closures contain all non steep functions with non-zero gradient at the point $I$. 
\newtheorem{sigmarn}[steepness]{Definition}
\begin{sigmarn}\label{bm}
Take two integers $r,n\geq 2$ and define $\beta_m:=\displaystyle\frac{\bar{\alpha}_m+3}{2}$, with $\bar{\alpha}_m$ as in Theorem \ref{nek2}.
\newline
 $\sigma_n^r(I) \subset\mathcal{P}_I(r,n)$  is the set containing the $r$-jets of smooth functions $h$ such that
\begin{enumerate}
\item $\nabla h(I)\neq 0$
\item There exists an $m$-dimensional subspace $\Lambda_m^I$ orthogonal to $\nabla h(I)$ and a curve $\gamma:\mathbb{R}\longrightarrow\Lambda_m^I$ of the form
\begin{equation}\label{curve}
\gamma(t):=\begin{cases}
x_1(t)=t\\
x_i(t)=\displaystyle\sum_{j=1}^{\beta_m-1}b_{ij}t^j\ ,\ \ i\in\{2,...,m\}\ ,\ \ b_{ij}\in\mathbb{R}
\end{cases}\ ,
\end{equation} 
such that the restriction of the gradient of $h$ to $\gamma(t)$ has a zero of order not smaller than $\beta_m-1$ at $t=0$:
\begin{equation}\label{zero}
\left.\frac{d^p (\nabla h|_{\Lambda_m^I})|_{\gamma(t)} }{dt^p}\right|_{t=0}=0\ ,\ \ p\in\{1,2,...,\beta_m-1\}\ .
\end{equation}  
\end{enumerate}
\end{sigmarn}

{\bfseries Remark.} The reader might wonder why the value $\beta_m=(\alpha_m+3)/2$ was chosen in the definition of $\sigma^r_n(I)$. Infact, in his first work on the genericity of steepness, Nekhoroshev proves that, for any fixed $\beta_m\in\mathbb{N},\,\beta_m> 1$, any polynomial $P\in\mathcal{P}_I(r,n)\backslash\sigma^r_n(I)$ is steep on the subspace $\Lambda_m^I$ with indices $\alpha_m=2(\beta_m-1)-1$, hence $\beta_m=(\alpha_m+3)/2$ (see Theorem C and Lemma 7.2.2 in ref. \cite{Nekhoroshev_1973}). 

With this definition, we can write down the algebraic conditions that the $r$-jet $P_I(h,r,n)$ of a smooth function $h$ must satisfy  in some $m$-dimensional subspace $\Lambda_m^I$ in order to belong to $\sigma_n^r(I)$.
For fixed $m$, these can be gathered in a system $\Xi_m(h,I,n)$ composed of four subsystems $\xi_{m,l}$, with $l=1,2,3,4$,
$$
\Xi_m(h,I,n):=
\begin{cases}
\xi_{m,1}(h)\ ;\ \
\xi_{m,2}(h,A^i)\\
\xi_{m,3}(h,A^i)\ ;\ \
\xi_{m,4}(h, A^i, b_{ij})
\end{cases}\ ,
$$
where $i\in\{1,...,m\},\ j\in\{1,...,\beta_m-1\}$, the $A^i$ are linearly independent vectors (with origin at $I$) which constitute a basis for $\Lambda_m^I$ and the coefficients $b_{ij}$ are real parameters defining a curve $\gamma(t)$ as in (\ref{curve}).
One has that
\begin{enumerate}
\item $\xi_1(h)$ imposes $\nabla h(I)\neq 0$;
\item $\xi_2(h,A^i)$ imposes the vectors $A^1,...,A^m$ to be linearly independent,
$$
rk[A^1,...,A^m]=m\ ;
$$  
\item $\xi_3(h,A^i)$ imposes the vectors $A^1,...,A^m$ to be orthogonal to $\nabla h(I)$,
$$
h_I^1[A^1]=0\ ;\ \
...\ ;\ \
h_I^1[A^m]=0\ ;
$$
\item $\xi_4(h,A^i,b_{ij})$ is a system of $m(\beta_m-1)$ equations obtained as follows. We denote with $x_1,...,x_m$ the coordinates for $\Lambda^I_m$ with respect to the basis $A^1,...,A^m$. By construction, such coordinates are null at $I$. Then, we consider the Taylor polynomial of $h|_{\Lambda_m^I}$ at $I$ up to order $\beta_m$, namely
\begin{align}
\begin{split}
P^{\beta_m}_n(x):=&\sum_{i=1}^m h^1_I[A^i]x_i+\frac{1}{2}\sum_{i,j=1}^m h^2_I[A^i,A^j]x_ix_j\\
+&...+\frac{1}{\beta_m!}\sum_{\underbrace{i,j,k,...,l=1}_{\beta_m\ \text{terms}}}^{m}h^{\beta_m-1}_I[A^i,A^j,A^k,...,A^l]x_ix_jx_k...x_l\ .
\end{split}
\end{align}
Condition (\ref{zero}) can now be imposed by considering the gradient $\nabla P^{\beta_m}_n(x)$, by injecting expression (\ref{curve}) in each of its $m$ components and by requiring that the $\beta_m-1$ coefficients of the resulting polynomial in $t$ are null. One thus obtains $m(\beta_m-1)$ equations.
\end{enumerate}
For fixed $m$, $\Xi_m(h,I,n)$ is said to be solvable for a given $h$ at $I$ if there exist a basis $A^1,...,A^m$ and real parameters $b_{ij}$ that verify it. $P_I(h,r,n)$ belongs to $\sigma_n^r(I)$ if at least one of the systems $\Xi_m(h,I,n)$, with $m\in\{1,...,n-1\}$, is solvable for $h$.
\newline
Indeed, following Theorem \ref{nek2}, in the sequel we will try to consider the closure of the algebraic conditions defining $\sigma_n^5(I)$ and, when this turns out to be too complicated, we will choose suitable closed sets containing $\sigma_n^5(I)$, with $n=2,3,4,5$. We will not deal with the case $n\geq 6$ since in such situation the conditions we find yield sets of steep functions which are smaller than those yielded by the three-jet non degeneracy condition, as it was already pointed out in ref. \cite{Schirinzi_Guzzo_2013}. 

\section{Proofs of Theorems \ref{n2}-\ref{n5}}\label{proofs}
For the sake of simplicity, from now on we drop the subscript $I$ in $h_I$ referring to the point where the considered jet is calculated. Moreover, we denote with $\Pi_{\Lambda^I_m}$ the projection onto an $m$-dimensional linear affine subspace $\Lambda_m^I$ orthogonal to the gradient. We start by stating the following simple lemma, which will turn out to be useful when trying to prove the closedness of the sets which we shall consider in the sequel, namely
\newtheorem{closedness}{Lemma}
\begin{closedness}\label{closedness}
	Let $E$ be a metric space, $K$ a compact subset of some metric space and $\Delta$ a closed subset of $E\times K$. Then, the projection of $\Delta$ on $E$, denoted with $\Pi_E(\Delta)$, is closed. 
\end{closedness}
\begin{proof}
	Let $\{p_n\}_{n\in\mathbb{N}}$ be a sequence in $\Pi_E(\Delta)$ converging to a point $\bar{p}$ and $\{k_n\}_{n\in\mathbb{N}}$ a sequence in $K$ satisying $(p_n,k_n)\in\Delta$. Since $K$ is a compact subset of some metric space, one can extract a subsequence $\{k_{n_l}\}_{l\in\mathbb{N}}$ converging to a point $\bar{k}\in K$. Hence, the sequence $\{(p_{n_l},k_{n_l})\}_{l\in\mathbb{N}}$ in $\Delta$ converges to $(\bar{p},\bar{k})\in\Delta$, since $\Delta$ is closed. This implies that $\bar{p}$ belongs to $\Pi_E(\Delta)$, which is therefore closed.
\end{proof}
The following sets will turn out to be particularly useful in the sequel.
\newtheorem{def2}[steepness]{Definition}
\newtheorem{def2bis}[steepness]{Definition}
\begin{def2}\label{Psi1star}
For $n=2,3,4$, we denote with $\Psi_1^*(n)\subset\mathcal P_I(5,n)$ the set of those jets of order five which satisfy the five-jet degeneracy condition. Similarly, for $n=5$ we denote with $\Psi_1^*(5)\subset\mathcal P_I(5,5)$ the set of those jets of order five which are four-jet degenerate. 
\newline
Moreover, for $n=1,2,3,4,5$, we indicate with $\Psi_1(n)\subset\mathcal P_I(5,n)$ the intersection between $\Psi_1^*(n)$ and the set containing those jets corresponding to functions having non-zero gradient at $I$. 
\end{def2} 
In particular, by Lemma 4.1 in ref. \cite{Schirinzi_Guzzo_2013} one has
\newtheorem{lem2}[closedness]{Lemma}
\begin{lem2}\label{lem2}
For $n=2,3,4,5$, the set $\Psi_1^*(n)$ is closed and it coincides with the closure of $\Psi_1(n)$.
\end{lem2}
With this setup, we are now ready to give the proofs of Theorems \ref{n2}-\ref{n5}.
\subsection{Proof of Theorem \ref{n2} (n=2)}
\begin{proof}
We assume the hypotheses of Theorem \ref{n2}. Since we are in a domain of $\mathbb{R}^2$, the only possible dimension for a subspace orthogonal to the gradient is $m=1$. For $n=2$ and $r=5$ we have $\beta_1=5$. Now, we build the set $\sigma_2^5(I)$ by following the strategy described by Nekhoroshev in \cite{Nekhoroshev_1979} and which we recalled in Theorem \ref{nek2} and Definition \ref{bm}.
First, we consider the Taylor polynomial of the restriction of the function $h$ to the subspace $\Lambda_1^I$ up to order $\beta_1=5$:
\begin{align}
\begin{split}
P^5_2(x)=&h^1[v]x+\frac{1}{2}h^2[v,v]x^2+\frac{1}{6}h^3[v,v,v]x^3\\
+&\frac{1}{24}h^4[v,v,v,v]x^4+\frac{1}{120}h^5[v,v,v,v,v]x^5\ ,
\end{split}
\end{align}
where $v$ is a non-null vector orthogonal to the gradient.
Then, we calculate $\nabla P_5^2(x)$ and we consider its restriction to the curve $x(t)=t$.
By setting all the coefficients of such polynomial to be equal to zero we obtain the subsystem $\xi_4(h,v)$ described in the previous section, so that system $\Xi_1(h,I)$ reads
\begin{equation}\label{5jetnd}
\begin{cases}
\nabla h(I)\neq 0\ ;\ \
v\neq 0\ ;\ \
h^1[v]=0\ ;\ \ 
h^2[v,v]=0\\
h^3[v,v,v]=0\ ;\ \
h^4[v,v,v,v]=0\ ;\ \
h^5[v,v,v,v,v]=0
\end{cases}\ .
\end{equation}
Since this is the only system we can consider in this case, we have that the set $\sigma_2^5(I)$ coincides with the one defined by  $\Xi_1(h,I)$ which, in turn, is equal to $\Psi_1(2)$ by Definition \ref{Psi1star}. Theorem \ref{n2} then follows from Lemma \ref{lem2} and Theorem \ref{nek2}.
\end{proof}
\subsection{Proof of Theorem \ref{n3} (n=3)}
Analogously to the case $n=2$, we give some suitable definitions.
\newtheorem{def3}[steepness]{Definition}
\newtheorem{def3bis}[steepness]{Definition}
\newtheorem{def3ter}[steepness]{Definition}
\newtheorem{def3quater}[steepness]{Definition}
\begin{def3bis}\label{def3bis}
We denote by $\Psi_2(3)$ the set in the space of $5$-jets of smooth functions $h$ of three variables such that there exist two linearly independent vectors $u,v$ and two real parameters $\alpha,\beta$ satisfying
\begin{equation}
\begin{cases}
\nabla h(I)\neq 0\ ;\ \ 
h^1[u]=
h^1[v]=
\Pi_{\Lambda_2^I}h^2[v,\cdot]=
\Pi_{\Lambda_2^I}\left(2\alpha h_I[u,\cdot]+h^3[v,v,\cdot]\right)=0\\
\Pi_{\Lambda_2^I} (  6\beta h^2[u,\cdot]+6\alpha h^3[u,v,\cdot]+  h_I^4[v,v,v,\cdot])=0\ .
\end{cases}
\end{equation}
\end{def3bis}
\begin{def3quater}
We denote by $\Psi_2^*(3)$ the set in the space of $5$-jets of smooth functions $h$ of three variables such that there exists two linearly independent vectors $u,v$ satisfying
\begin{equation}\label{sii}
\begin{cases}
h^1[u]=
h^1[v]=
h^2[v,v]=
h^2[v,u]=
h^3[v,v,v]=0\\
h^2[u,u]h^4[v,v,v,v]=3(h^3[v,v,u])^2\\
\end{cases}\ .
\end{equation}
\end{def3quater}
The following result holds true
\newtheorem{lem4}[closedness]{Lemma}
\begin{lem4}\label{lem4}
The set $\Psi_2^*(3)$ is closed and contains the closure of $\Psi_2(3)$.
\end{lem4}
\begin{proof}
We notice that all equations in (\ref{sii}) are homogeneous in $u$ and $v$, so that without any loss of generality we can consider $(u,v)\in \mathbb{S}^2\times \mathbb{S}^2$. Moreover, still without any loss of generality we can assume $u\cdot v=0$, since it is easy to see that the component of $u$ which is parallel to $v$ yields a null contribution to the system in \eqref{sii}. Then, system (\ref{sii}) defines an algebraic closed set in $\mathcal{P}_I(5,3)\times \mathbb{S}^2\times \mathbb{S}^2$, whose projection onto $\mathcal{P}_I(5,3)$ is $\Psi_2^*(3)$. Hence $\Psi_2^*(3)$ is closed by Lemma \ref{closedness}.
In order to prove inclusion, we write the system defining $\Psi_2(3)$ in its less compact form
\begin{equation}\label{siiiii}
\begin{cases}
\nabla h(I)\neq 0\ ;\ \
h^1[u]=0\ ;\ \
h^1[v]=0\ ;\ \
h^2[v,v]=0\\
h^2[v,u]=0\ ;\ \
h^3[v,v,v]=0\ ;\ \
6\alpha h^3[u,v,v]+  h^4[v,v,v,v]=0\\
2\alpha h^2[u,u]+h^3[u,v,v]=0\\
6\beta h^2[u,u]+6\alpha h^3[u,u,v]+  h^4[v,v,v,u]=0\\
\end{cases}\ .
\end{equation}
By applying Gauss elimination method to the last two equations
and by subtracting one to another, one can get rid of parameter $\alpha$ and obtains
$$
3(h^3[u,v,v])^2=h^4[v,v,v,v]h^2[u,u]\ .
$$
Then, by discarding the last equation and the inequality on the gradient of $h$, one reduces to the system defining the set $\Psi_2^*(3)$. Therefore the inclusion 
$
\Psi_2^*(3)\supset \Psi_2(3)
$
holds. 
Since $\Psi_2^*(3)$ is closed, one has $\Psi_2^*(3)\supset \bar{\Psi}_2(3)$ and the statement is thus proven.
\end{proof}
We remark that we considered set $\Psi_2^*(3)$ since $\Psi_2(3)$ is not closed, as we show in the following
\newtheorem{ex3}[ex1]{Example}
\begin{ex3}
	For $k\in\mathbb{N}$, consider the sequence of polynomial functions $$h_k(I_1,I_2,I_3)=\frac32 \frac{I_1^4+I_2^4}{4!}-\frac{I_3^4}{4!k}-\frac{I_2I_1^2}{2k}+\frac{I_2^2}{2k^2}+I_3\ ,$$
	converging to 
	$h(I_1,I_2,I_3)=\displaystyle \frac32 \frac{I_1^4+I_2^4}{4!}+I_3\ .$
	At the origin, the jet $P(h_k,5,3)$ associated to $h_k$ belongs to the set $\Psi_2(3)$ for all $k$, but the jet $P(h,5,3)$ associated to the limit function does not.
\end{ex3}
\begin{proof}
	For fixed $k\in\mathbb{N}$, set $\alpha_k:=\displaystyle\frac{k}{2}$, $\beta_k:=\displaystyle\frac{k^2}{2}$ and the vectors $u=(0,1,0)$, $v=(1,0,0)$. It is straightforward to see that $P(h_k,5,3)\in \Psi_2(3)$ at the origin, with such choice of vectors and parameters. However, the limit function $h$ is weakly-convex at the origin and, as the reader can easily verify, it does not fulfill system (\ref{siiiii}) for any non-null vector $v$. 
\end{proof}
We are now ready to write the proof of Theorem \ref{n3}.
\begin{proof}
We assume the hypotheses of Theorem \ref{n3}. Since we are in a domain of $\mathbb{R}^3$, $m$ can be equal to $1$ or $2$. For $n=3$ and $r=5$ we have $\beta_1=5$ and $\beta_2=4$. 
\newline
For $m=1$, by following the same construction as in the case $n=2$, we have the same expression of (\ref{5jetnd}) for $\Xi_1(h,I,3)$.
\newline
In order to build up system $\Xi_2(h,I,3)$, we follow the usual strategy described by Nekhoroshev in \cite{Nekhoroshev_1979} and we consider the Taylor polynomial of the restriction of the function $h$ to the subspace $\Lambda_2^I$ up to order $\beta_2=4$.
By calculating $\nabla P_4^3(x)=(\partial_{x_1}P_4^3(x),\partial_{x_2}P_4^3(x))$ along the curve 
$$
x_1=t\ ;\ \
x_2= b_{21} t+b_{22}t^2+b_{23}t^3
$$
and by setting equal to zero all the coefficients of the resulting polynomial in $t$ up to order $\beta_2-1=3$, one has that
\begin{enumerate}
\item The linear terms yield 
$
\Pi_{\Lambda_2^I}h^2[A^1+b_{21}A^2,\cdot]=0\ ;
$
\item The quadratic terms yield 
\begin{equation}
\Pi_{\Lambda_2^I}\left(2b_{22}h^2[A^2,\cdot]+h^3[A^1+b_{21}A^2,A^1+b_{21}A^2,\cdot]\right)=0\ ;
\end{equation}
\item Finally, the cubic terms yield
\begin{align}
\begin{split}
\Pi_{\Lambda_2^I} ( & 6b_{23} h^2[A^2,\cdot]+6b_{22}h^3[A^2,A^1+b_{21}A^2,\cdot]\\+ & h^4[A^1+b_{21}A^2,A^1+b_{21}A^2,A^1+b_{21}A^2,\cdot])=0\ ;
\end{split}
\end{align}
\end{enumerate}
where $A^1,A^2$ are a basis for $\Lambda_2^I$.
\newline
Thus, system $\Xi_2(h,I,3)$ takes the form
\begin{equation}
\begin{cases}
\nabla h(I)\neq 0\ ;\ \
h^1[u]=
h^1[v]=
\Pi_{\Lambda_2^I}h^2[v,\cdot]=
\Pi_{\Lambda_2^I}\left(2\alpha h^2[u,\cdot]+h^3[v,v,\cdot]\right)=0\\
\Pi_{\Lambda_2^I} (  6\beta h^2[u,\cdot]+6\alpha h^3[u,v,\cdot]+  h^4[v,v,v,\cdot])=0\ ,
\end{cases}
\end{equation}
with $u:=A^2,v:=A^1+b_{21}A^2$ two linearly independent vectors and $\alpha:=b_{22},\beta:=b_{23}$ two real parameters.
With the help of Definitions \ref{Psi1star} and \ref{def3bis} we see that $\sigma_3^5(I)=\Psi_1(3)\cup\Psi_2(3)$. As a consequence of Lemmas \ref{lem2} and \ref{lem4} and of Theorem \ref{nek2} one has 
$
\Sigma_3^5(I)=\bar{\sigma}_3^5(I)\subset\Psi^*_1(3)\cup\Psi_2^*(3)
$
and Theorem \ref{n3} follows.
\end{proof}
\subsection{Proof of Theorem \ref{n4} (n=4)}
We start with the usual definitions
\newtheorem{def4ter}[steepness]{Definition}
\newtheorem{def4quater}[steepness]{Definition}
\newtheorem{def4quinquies}[steepness]{Definition}
\newtheorem{def4sesquies}[steepness]{Definition}
\begin{def4ter}\label{def4ter}
We denote by $\Psi_2(4)$ the set in the space of $5$-jets of smooth functions $h$ of four variables such that there exist two linearly independent vectors $u,v$ and three real parameters $\alpha,\beta,\gamma$ satisfying
\begin{equation}\label{sistemaalfann}
\begin{cases}
\nabla h(I)\neq 0\ ;\ \
h^1[u]=
h^1[v]=
\Pi_{\Lambda_2^I}h^2[v,\cdot]=
\Pi_{\Lambda_2^I}\left(2\alpha h^2[u,\cdot]+h^3[v,v,\cdot]\right)=0\\
\Pi_{\Lambda_2^I} (  6\beta h^2[u,\cdot]+6\alpha h^3[u,v,\cdot]+  h^4[v,v,v,\cdot])=0\\
\Pi_{\Lambda_2^I}(24\gamma h^2[u,\cdot]+24\beta h^3[u,v,\cdot]\\
+12\alpha^2h^3[u,u,\cdot]+12\alpha h^4[v,v,u,\cdot]+h^5[v,v,v,v,\cdot])=0
\end{cases}\ .
\end{equation}
\end{def4ter}
\begin{def4quater}\label{def4quater}
We denote by $\Psi_2^*(4)$ the set in the space of $5$-jets of smooth functions $h$ of four variables such that there exist two linearly independent vectors $u,v$ satisfying
\begin{equation}\label{sistemino}
\begin{cases}
h^1[u]=
h^1[v]=
h^2[v,v]=
h^2[v,u]=
h^3[v,v,v]=0\\
h^2[u,u]h^4[v,v,v,v]=3(h^3[v,v,u])^2\\
15(h^3[v,v,u])^2h^3[u,u,v]+h^5[v,v,v,v,v](h^2[u,u])^2\\
\ \ =10h^4[v,v,v,u]h^3[u,v,v]h^2[u,u]
\end{cases}\ .
\end{equation}
\end{def4quater}
\begin{def4quinquies}\label{def4quinquies}
We denote by $\Psi_3(4)$ the set in the space of $5$-jets of smooth functions $h$ of four variables such that there exist three linearly independent vectors $u,v,w$ and two real parameters $\alpha,\beta$ satisfying
\begin{equation}
\begin{cases}
\nabla h(I)\neq 0\ ;\ \
h^1[u]=
h^1[v]=
h^1[w]=
\Pi_{\Lambda_3^I}h^2[v,\cdot]=0\\
\Pi_{\Lambda_3^I}(2h^2[\alpha u+\beta w,\cdot]+h^3[v,v,\cdot])=0\\
\end{cases}\ .
\end{equation}
\end{def4quinquies}
\begin{def4sesquies}\label{def4sesquies}
We denote by $\Psi_3^*(4)$ the set in the space of $5$-jets of smooth functions $h$ of four variables such that there exist three linearly independent vectors $u,v,w$ satisfying
\begin{equation}\label{prova1}
h^1[u]=
h^1[v]=
h^1[w]=
h^2[v,v]=h^2[v,u]=h^2[v,w]=0\ ;\ \
h^3[v,v,v]=0
\ .
\end{equation}
\end{def4sesquies}
With these definitions, we have the following
\newtheorem{lem6}[closedness]{Lemma}
\begin{lem6}\label{lem6}
The sets $\Psi_2^*(4),\Psi_3^*(4)$ are closed and one also has the inclusions
$
\Psi_2^*(4)\supseteq\bar{\Psi}_2(4),
\Psi_3^*(4)\supseteq\bar{\Psi}_3(4)\ .
$
\end{lem6}
\begin{proof}
The proof is similar to that of Theorem \ref{lem4}: without any loss of generality, one can always choose the vectors to be perpendicular and unitary, so that systems (\ref{sistemino}) and (\ref{prova1}) define algebraic closed sets in $\mathcal{P}_I(5,n)\times\mathbb{S}^2\times\mathbb{S}^2$ and $\mathcal{P}_I(5,n)\times\mathbb{S}^3\times\mathbb{S}^3$, whose projections onto $\mathcal{P}_I(5,n)$ are $\Psi_2^*(4)$ and $\Psi_3^*(4)$, which are therefore closed thanks to Lemma \ref{closedness}.
As for the inclusions, the relation
$
\Psi_3(4)\subset\Psi_3^*(4)
$
is immediate from definition \ref{def4quinquies}, once one projects the equations on the basis $u,v,w$ and compares the system to the one in definition \ref{def4sesquies}.
\newline
In order to prove that
$
\Psi_2(4)\subset\Psi_2^*(4)\ ,
$
we consider system (\ref{sistemaalfann}) defining $\Psi_2(4)$ in its most explicit form
\begin{equation}\label{sit}
\begin{cases}
\nabla h(I)\neq 0\ ;\ \
h^1[u]=h^1[v]=0\ ;\ \
h^2[v,v]=h^2[v,u]=0\\
h^3[v,v,v]=0\ ;\ \
2\alpha h^2[u,u]+h^3[v,v,u]=0\\
6\alpha h^3[u,v,v]+  h^4[v,v,v,v]=0\\
6\beta h^2[u,u]+6\alpha h^3[u,u,v]+  h^4[v,v,v,u]=0\\
24\beta h^3[u,v,v]
+12\alpha^2h^3[u,u,v]+12\alpha h^4[v,v,v,u]+h^5[v,v,v,v,v]=0\\
24\gamma h^2[u,u]+24\beta h^3[u,u,v]\\
+12\alpha^2h^3[u,u,u]+12\alpha h^4[v,v,u,u]+h^5[v,v,v,v,u]=0
\end{cases}\ .
\end{equation}
Applying Gauss elimination method in order to get rid of parameters $\alpha, \beta$ and discarding the last equation containing parameter $\gamma$ yields system (\ref{sistemino}) defining $\Psi_2^*(4)$.
Therefore, one has 
$
\Psi_2(4)\subset\Psi_2^*(4)\ .
$
Since $\Psi_2^*(4)$ and $\Psi_3^*(4)$ are closed, one finally obtains
$
\Psi_2^*(4)\supseteq\bar{\Psi}_2(4)\ ,
\Psi_3^*(4)\supseteq\bar{\Psi}_3(4)\ .
$
\end{proof}
With this setup, we are ready to prove Theorem \ref{n4}.
\begin{proof}
Since we work in a domain of $\mathbb{R}^4$, $m$ can be equal to $1,2$ or $3$. For $n=4$ and $r=5$ we have $\beta_1=\beta_2=5$ and $\beta_3=3$. 
\newline
For $m=1$, we follow the same construction as in the cases $n=2,3$ and system $\Xi_1(h,I,4)$ defines  a set $\Psi_1(4)$ whose closure concides with $\Psi_1^*(4)$. 
In order to build up system $\Xi_2(h,I,4)$, we follow once again the construction in \cite{Nekhoroshev_1979} and we consider the Taylor polynomial of the restriction of the function $h$ to a subspace $\Lambda_2^I$ up to order $\beta_2=5$.
By calculating $\nabla P_5^4(x)=(\partial_{x_1}P_5^4(x),\partial_{x_2}P_5^4(x))$ along the curve 
$$
x_1=t\ ;\ \
x_2= b_{21} t+b_{22}t^2+b_{23}t^3+b_{24}t^4
$$
and by setting equal to zero all the coefficients of the resulting polynomial in $t$ up to order $\beta_2-1=4$, one has that
\begin{enumerate}
\item The linear terms yield 
$
\Pi_{\Lambda_2^I}h^2[A^1+b_{21}A^2,\cdot]=0\ ;
$
\item The quadratic terms yield 
\begin{equation}
\Pi_{\Lambda_2^I}\left(2b_{22}h^2[A^2,\cdot]+h^3[A^1+b_{21}A^2,A^1+b_{21}A^2,\cdot]\right)=0\ ;
\end{equation}
\item The cubic terms yield
\begin{align}
\begin{split}
\Pi_{\Lambda_2^I} ( & 6b_{23} h^2[A^2,\cdot]+6b_{22}h^3[A^2,A^1+b_{21}A^2,\cdot]\\+ & h^4[A^1+b_{21}A^2,A^1+b_{21}A^2,A^1+b_{21}A^2,\cdot])=0\ ;
\end{split}
\end{align}
\item The quartic terms yield
\begin{align}
\begin{split}
\Pi_{\Lambda_2}(& 24b_{24} h^2[A^2,\cdot] + 
 24 b_{23} h^3[A^1+b_{21}A^2,A^2,\cdot] + 12 b_{22}^2 h^3[A^2,A^2,\cdot]\\
 + & 12 b_{22}h^4[A^1+b_{21}A^2,A^1+b_{21}A^2,A^2,\cdot]\\
  + & h^5[A^1+b_{21}A^2,A^1+b_{21}A^2,A^1+b_{21}A^2,A^1+b_{21}A^2,\cdot])=0\ ;\\
\end{split}
\end{align}
\end{enumerate}
where $A^1,A^2$ are a basis for $\Lambda_2^I$.
Thus, by following the same strategy as in the previous section, $\Xi_2(h,I)$ has the form
\begin{equation}
\begin{cases}
 \nabla h(I)\neq 0\ ;\ \
 h^1[u]=
 h^1[v]=0\ ;\ \
 \Pi_{\Lambda_2^I} h^2[v,\cdot]=0\\
\Pi_{\Lambda_2^I} \left( 2\alpha h^2[u,\cdot]+h^3[v,v,\cdot]\right)=0\\
 \Pi_{\Lambda_2^I} ( 6\beta h^2[u,\cdot]+6\alpha h^3[u,v,\cdot]+  h^4[v,v,v,\cdot])=0\\
 \Pi_{\Lambda_2}(24\gamma h^2[u,\cdot] + 
 24 \beta h^3[v,u,\cdot] + 12 \alpha^2 h^3[u,u,\cdot]\\
 + 12 \alpha h^4[v,v,u,\cdot]
+  h^5[v,v,v,v,\cdot])=0\ .\\
\end{cases}
\end{equation}
with $u:=A^2,v:=A^1+b_{21}A^2$ two linearly independent vectors and $\alpha:=b_{22},\beta:=b_{23},\gamma:=b_{24}$ three real parameters.
Finally, we construct system $\Xi_3(h,I)$.
We consider the Taylor polynomial $P_5^4(x)$ of the restriction of the function $h$ to the subspace $\Lambda_3^I$, up to order $\beta_3=3$.
By calculating $\nabla P_5^4(x)=(\partial_{x_1}P_5^4(x),\partial_{x_2}P_5^4(x),\partial_{x_3}P_5^4(x))$ along the curve 
$$
x_1=t\ ;\ \
x_2= b_{21} t+b_{22}t^2\ ;\ \
x_3= b_{31} t+b_{32}t^2
$$
and by setting equal to zero all the coefficients of the resulting polynomial in $t$ up to order $\beta_3-1=2$, one has that
\begin{enumerate}
\item The linear terms yield 
$
\Pi_{\Lambda_3^I}(h^2[A^1+b_{21}A^2+b_{31}A^3,\cdot])=0\ ;
$
\item The quadratic terms yield 
\begin{align}
\begin{split}
\Pi_{\Lambda_3^I}( & h^2[2b_{22}A^2+2b_{32}A^3,\cdot]\\
+&h^3[A^1+b_{21}A^2+b_{31}A^3,A^1+b_{21}A^2+b_{31}A^3,\cdot])=0\ ;
\end{split}
\end{align}
where $A_1,A_2,A_3$ are a basis for $\Lambda_3^I$.
\end{enumerate}
Thus, by following the same strategy as in the previous section, $\Xi_3(h,I)$ reads
\begin{equation}
\begin{cases}
\nabla h(I)\neq 0\ ;\ \
 h^1[u]=0\ ;\ \
 h^1[v]=0\ ;\ \
 h^1[w]=0\\
 \Pi_{\Lambda_3^I} h^2[v,\cdot]=0\ ;\ \
\Pi_{\Lambda_3^I} \left( 2 h^2[\alpha u+\beta w,\cdot]+h^3[v,v,\cdot]\right)=0\\
\end{cases}\ .
\end{equation}
with $u:=A^2,v:=A^1+b_{21}A^2+b_{31}A^3$, $w:=A^3$ three linearly independent vectors and $\alpha:=b_{22},\beta:=b_{32}$ two real parameters.
With the help of Definitions \ref{Psi1star}, \ref{def4ter} and \ref{def4quinquies}, we see that
$
\sigma_4^5(I)=\Psi_1(4)\cup\Psi_2(4)\cup\Psi_3(4)\ ,
$
so that, as a consequence of Lemma \ref{lem6} and of Theorem \ref{nek2}, one has 
$
\Sigma_4^5(I)=\bar{\sigma}_4^5(I)\subset\Psi^*_1(4)\cup\Psi_2^*(4)\cup\Psi_3^*(4),
$ which, together with Theorem \ref{nek2} once again, implies Theorem \ref{n4}.
\end{proof}
\subsection{Proof of Theorem \ref{n5} (n=5)}
We start with the usual definitions
\newtheorem{def5ter}[steepness]{Definition}
\newtheorem{def5quater}[steepness]{Definition}
\newtheorem{def5quinquies}[steepness]{Definition}
\newtheorem{def5sesquies}[steepness]{Definition}
\newtheorem{def5septies}[steepness]{Definition}
\newtheorem{def5octies}[steepness]{Definition}
\begin{def5ter}\label{def5ter}
We denote by $\Psi_2(5)$ the set in the space of $5$-jets of smooth functions $h$ of five variables such that there exist two linearly independent vectors $u,v$ and three real parameters $\alpha,\beta,\gamma$ satisfying
\begin{equation}
\begin{cases}
\nabla h(I)\neq 0\ ;\ \
h^1[u]=
h^1[v]=
\Pi_{\Lambda_2^I}h^2[v,\cdot]=
\Pi_{\Lambda_2^I}\left(2\alpha h^2[u,\cdot]+h^3[v,v,\cdot]\right)=0\\
\Pi_{\Lambda_2^I} (  6\beta h^2[u,\cdot]+6\alpha h^3[u,v,\cdot]+  h^4[v,v,v,\cdot])=0\\
\Pi_{\Lambda_2^I}(24\gamma h^2[u,\cdot]+24\beta h^3[u,v,\cdot]\\
+12\alpha^2h^3[u,u,\cdot]+12\alpha h^4[v,v,u,\cdot]+h^5[v,v,v,v,\cdot])=0
\end{cases}\ .
\end{equation}
\end{def5ter}
\begin{def5quater}\label{def5quater}
We denote by $\Psi_2^*(5)$ the set in the space of $5$-jets of smooth functions $h$ of five variables such that there exist two linearly independent vectors $u,v$ satisfying
\begin{equation}
\begin{cases}
h^1[u]=
h^1[v]=
h^2[v,v]=
h^2[v,u]=
h^3[v,v,v]=0\\
h^2[u,u]h^4[v,v,v,v]=3(h^3[v,v,u])^2\\
15(h^3[v,v,u])^2h^3[u,u,v]+h^5[v,v,v,v,v](h^2[u,u])^2\\
\ \ =10h^4[v,v,v,u]h^3[u,v,v]h^2[u,u]
\end{cases}\ .
\end{equation}
\end{def5quater}
\begin{def5quinquies}\label{def5quinquies}
We denote by $\Psi_3(5)$ the set in the space of $5$-jets of smooth functions $h$ of five variables such that there exist three linearly independent vectors $u,v,w$ and four real parameters $\alpha,\beta,\gamma,\delta$ satisfying
\begin{equation}
\begin{cases}
\nabla h(I)\neq 0\ ;\ \
h^1[u]=
h^1[v]=
h^1[w]=
\Pi_{\Lambda_3^I}h^2[v,\cdot]=0\\
\Pi_{\Lambda_3^I}(h^2[\alpha u+\beta w,\cdot]+h^3[v,v,\cdot])=0\\
\Pi_{\Lambda_3^I}(6h^2[\gamma u +\delta w,\cdot]+6h^3[\alpha u +\beta w,v,\cdot]+h^4[v,v,v,\cdot])=0
\end{cases}\ .
\end{equation}
\end{def5quinquies}
\begin{def5sesquies}\label{def5sesquies}
We denote by $\Psi_3^*(5)$ the set in the space of $5$-jets of smooth functions $h$ of five variables such that there exist three linearly independent vectors $u,v,w$ satisfying
\begin{equation}
\begin{cases}
h^1[u]=h^1[v]=h^1[w]=
h^2[v,v]=h^2[v,u]=h^2[v,w]=
h^3[v,v,v]=0\\
12h^3[u,v,v]h^3[v,v,w]h^2[u,u]h^2[u,w]\\
\ \ -6(h^3[u,v,v])^2(h^2[u,w])^2-6(h^3[v,v,w])^2(h^2[u,u])^2\\
\ \ +\{h^4[v,v,v,v]h^2[u,u]-6(h^3[u,v,v])^2\}\{h^2[w,w]h^2[u,u]-(h^2[u,w])^2\}=0
\end{cases}\ .
\end{equation}
\end{def5sesquies}
\begin{def5septies}\label{def5septies}
We denote by $\Psi_4(5)$ the set in the space of $5$-jets of smooth functions $h$ of five variables such that there exist four linearly independent vectors $u,v,w,x$ satisfying
\begin{equation}
\nabla h(I)\neq 0\ ;\ \
h^1[u]=h^1[v]=h^1[w]=h^1[x]=0\ ;\ \
\Pi_{\Lambda_4^I}h^2[v,\cdot]=0\ .
\end{equation}
\end{def5septies}
\begin{def5septies}\label{def5octies}
We denote by $\Psi_4^*(5)$ the set in the space of $5$-jets of smooth functions $h$ of five variables such that there exist four linearly independent vectors $u,v,w,x$ satisfying
\begin{equation}
h^1[u]=h^1[v]=h^1[w]=h^1[x]=0\ ;\ \
\Pi_{\Lambda_4^I}h^2[v,\cdot]=0\ .
\end{equation}
\end{def5septies}
We have the following result:
\newtheorem{lem8}[closedness]{Lemma}
\begin{lem8}\label{lem8}
The sets $\Psi_2^*(5),\Psi_3^*(5),\Psi_4^*(5)$ are closed and the following inclusions hold:
$
\Psi_2^*(5)\supset\bar{\Psi}_2(5),
\Psi_3^*(5)\supset\bar{\Psi}_3(5),
\Psi_4^*(5)\supset\bar{\Psi}_4(5).
$
\end{lem8}
\begin{proof}
Closure of the three sets $\Psi_2^*(5),\Psi_3^*(5),\Psi_4^*(5)$ is proven exactly in the same way as in the previous paragraphs, with the help of Lemma \ref{closedness}.
The proof of the inclusion $\Psi_2^*(5)\supset\bar{\Psi}_2(5)$ is identic to the one given in Lemma \ref{lem6} for the inclusion $\Psi_2^*(4)\supset\bar{\Psi}_2(4)$.
Inclusion $\Psi_4^*(5)\supset\bar{\Psi}_4(5)$ is immediate when considering the definitions of $\Psi_4(5)$ and $\Psi_4^*(5)$ and the closure of the latter.
The only non-trivial inclusion is thus $\Psi_3^*(5)\supset\bar{\Psi}_3(5)$. In order to prove it, we rewrite the system defining $\Psi_3(5)$ in its less synthetic form
\begin{equation}\label{sist5}
\begin{cases}
\nabla h(I)\neq 0\ ;\ \
h^1[u]=
h^1[v]=
h^1[w]=
h^2[v,v]=h^2[v,u]=h^2[v,w]=0\\
h^3[v,v,v]=0\ ;\ \
h^2[\alpha u+\beta w,u]+h^3[v,v,u]=0\\
h^2[\alpha u+\beta w,w]+h^3[v,v,w]=0\\
6h^2[\gamma u +\delta w,u]+6h^3[\alpha u +\beta w,v,u]+h^4[v,v,v,u]=0\\
6h^2[\gamma u +\delta w,w]+6h^3[\alpha u +\beta w,v,w]+h^4[v,v,v,w]=0\\
6h^3[\alpha u +\beta w,v,v]+h^4[v,v,v,v]=0\\
\end{cases}\\ \ .
\end{equation}
Once again, Gauss elimination method can be used in order to get rid of parameters $\alpha$ and $\beta$. Then, by discarding the first inequality and the two equations containing $\gamma,\delta$ in system (\ref{sist5}) defining $\Psi_3(5)$, one obtains the system in Definition \ref{def4sesquies}, which determines $\Psi_3^*(5)$. Therefore, one has $\Psi_3^*(5)\supset\Psi_3(5)$ and $\Psi_3^*(5)\supset\bar{\Psi}_3(5)$ since $\Psi^*_3(5)$ is closed.
\end{proof}
With this background, we are ready to prove Theorem \ref{n5}.
\begin{proof}
Since we work in a domain of $\mathbb{R}^5$, $m$ can be equal to $1,2,3$ or $4$. For $n=5$ and $r=5$ we have $\beta_1=4, \beta_2=5,\beta_3=4$ and $\beta_4=2$. 
\newline
For $m=1$, by following the same construction as in the cases $n=2,3,4$, we find the following expression for $\Xi_1(h,I,5)$:
\begin{equation}
\nabla h(I)\neq 0\ ;\ \
v\neq 0\ ;\ \
h^1[v]=
h^2[v,v]=
h^3[v,v,v]=
h^4[v,v,v,v]=0
\ ,
\end{equation}
so that $\Xi_1(h,I,5)=\Psi_1(5)$ by Definition \ref{Psi1star}.\newline
Now, since $\beta_2=5$ as it was in the case $n=4$, we have exactly the same construction and we can write $\Xi_2(h,I,5)$ in the same form:
\begin{equation}
\begin{cases}
\nabla h(I)\neq 0\ ;\ \
 h^1[u]=
 h^1[v]=
 \Pi_{\Lambda_2^I} h^2[v,\cdot]=
\Pi_{\Lambda_2^I} \left( 2\alpha h^2[u,\cdot]+h^3[v,v,\cdot]\right)=0\\
 \Pi_{\Lambda_2^I} ( 6\beta h^2[u,\cdot]+6\alpha h^3[u,v,\cdot]+  h^4[v,v,v,\cdot])=0\\
 \Pi_{\Lambda_2}(24\gamma h^2[u,\cdot] + 
 24 \beta h^3[v,u,\cdot] + 12 \alpha^2 h^3[u,u,\cdot]\\
 + 12 \alpha h^4[v,v,u,\cdot]
+  h^5[v,v,v,v,\cdot])=0\ .\\
\end{cases}
\end{equation}
with $u:=A^2,v:=A^1+b_{21}A^2$ two linearly independent vectors and $\alpha:=b_{22},\beta:=b_{23},\gamma:=b_{24}$ three real parameters.
We now construct $\Xi_3(h,I,5)$.
As usual, we consider the Taylor polynomial $P_5^4(x_1,x_2,x_3)$ of the restriction of the function $h$ to the subspace $\Lambda_3^I$ up to order $\beta_3=4$.
By calculating $\nabla P_5^4(x)=(\partial_{x_1}P_5^4(x),\partial_{x_2}P_5^4(x),\partial_{x_3}P_5^4(x))$ along the curve 
$$
x_1=t\ ;\ \
x_2= b_{21} t+b_{22}t^2+b_{23}t^3\ ;\ \
x_3= b_{31} t+b_{32}t^2+b_{24}t^4
$$
and by setting equal to zero all the coefficients of the resulting polynomial in $t$ up to order $\beta_3-1=3$, one has that
\begin{enumerate}
\item The linear terms yield
$
\Pi_{\Lambda_3^I}(h^2[A^1+b_{21}A^2+b_{31}A^3,\cdot])=0\ ;
$
\item The quadratic terms yield 
\begin{align}
\begin{split}
\Pi_{\Lambda_2^I} \left(\right. &h^2[b_{22}A^2+b_{32}A^3,\cdot] \\
 &+h^3[A^1+b_{21}A^2+b_{31}A^3,A^1+b_{21}A^2+b_{31}A^3,\cdot]\left.\right)=0\ ;
 \end{split}
\end{align} 
\item The cubic terms yield 
\begin{align}
\begin{split}
\Pi_{\Lambda_3^I}( & 6h^2[2b_{23}A^2+2b_{33}A^3,\cdot]
+6 h^3[b_{22}A^2+b_{23}A^3,A^1+b_{21}A^2+b_{31}A^3,\cdot]\\
+ &h^4[A^1+b_{21}A^2+b_{31}A^3,A^1+b_{21}A^2+b_{31}A^3,A^1+b_{21}A^2+b_{31}A^3,\cdot])\\
&=0\ ;
\end{split}
\end{align}
where $A_1,A_2,A_3$ are a basis for $\Lambda_3^I$.
\end{enumerate}
Therefore, $\Xi_3(h,I,5)$ can be compactly formulated as
\begin{equation}
\begin{cases}
\nabla h(I)\neq 0\ ;\ \
h^1[u]=
h^1[v]=
h^1[w]=0\ ;\ \
\Pi_{\Lambda_3^I}h^2[v,\cdot]=0\\
\Pi_{\Lambda_3^I}(h^2[\alpha u+\beta w,\cdot]+h^3[v,v,\cdot])=0\\
\Pi_{\Lambda_3^I}(6h^2[\gamma u +\delta w]+6h^3[\alpha u +\beta w,v,\cdot]+h^4[v,v,v,\cdot])=0
\end{cases}\ .
\end{equation}
with $u:=A^2,v:=A^1+b_{21}A^2+b_{31}A^3$, $w:=A^3$ three linearly independent vectors and $\alpha:=b_{22},\beta:=b_{32}, \gamma:=b_{23},\delta:=b_{33}$ four real parameters.
\newline
Finally, we construct $\Xi_4(h,I,5)$ in the same usual way.
\newline
As usual, we consider the Taylor polynomial $P_5^4(x_1,x_2,x_3)$ of the restriction of the function $h$ to the subspace $\Lambda_4^I$ up to order $\beta_4=2$.
\newline
By calculating $\nabla P_5^4(x)=(\partial_{x_1}P_5^4(x),\partial_{x_2}P_5^4(x),\partial_{x_3}P_5^4(x))$ along the curve 
$$
x_1(t)=t\ ;\ \
x_2(t)= b_{21} t\ ;\ \
x_3(t)= b_{31} t\ ;\ \
x_4(t)=b_{41} t
$$
and by setting equal to zero all the coefficients of the resulting polynomial in $t$ up to order $\beta_4-1=1$, $\Xi_4(h,I,5)$ reads
\begin{equation}
\nabla h(I)\neq 0\ ;\ \
h^1[u]=h^1[v]=h^1[w]=h^1[x]=0\ ;\ \
\Pi_{\Lambda_4}h^2[v,\cdot]=0
\ ,
\end{equation}
with $v=A^1+b_{21}A^2+b_{31}A^3+b_{41}A^4,u=A^2,w=A^3,x=A^4$.
With the help of definitions \ref{Psi1star}, \ref{def5quater} and \ref{def5sesquies}, we see that 
$
\sigma_5^5(I)=\Psi_1(5)\cup\Psi_2(5)\cup\Psi_3(5)\cup\Psi_4(5)
$
so that, as a consequence of Lemma \ref{lem8} and of Theorem \ref{nek2}, one has 
$
\Sigma_5^5(I)=\bar{\sigma}_5^5(I)\subset\Psi^*_1(5)\cup\Psi_2^*(5)\cup\Psi_3^*(5)\cup\Psi_4^*(5)
.$
This, together with Theorem \ref{nek2}, implies Theorem \ref{n5}.
\end{proof}
\section{Final remarks}\label{discussion}
\subsection{The case $n\geq 6$}\label{ngeq6}
As the computations in the previous sections showed (see e.g. the case $n=2$ or ref. \cite{Schirinzi_Guzzo_2013}), Nekhoroshev's construction on affine linear subspaces of dimension $m=1$ always yields a subsystem $\Xi_1(h,I,n)$ requiring $\beta_1$-degeneracy condition. In other words, take an arbitrary integer $r\geq 2$ and compute coefficient $\beta_1$ on a one-dimensional subspace; if there exists $v\neq 0$ such that
\begin{equation}
\nabla h(I)\neq 0\ ;\ \
h^1[v]=0\ ;\ \
h^2[v,v]=0\ ;\ \
...\ ;\ \
h^{\beta_1}[v,...,v]=0
\end{equation}
is satisfied, then the $r$-jet of $h$ belongs to $\sigma_n^r(I)$, since it fulfills membership requirements on subspaces of dimension $m=1$. On the other hand, algebraic conditions for steepness on jets of order strictly greater than three make sense only at those points $I$ where the function $h$ is three-jet degenerate, since three-jet non-degeneracy automatically implies steepness. By looking at the explicit expression for $\beta_m$ in Definition \ref{bm} and by taking expression (\ref{boundam}) for the maximal index of steepness $\bar{\alpha}_m$ into account, one easily sees that $\beta_1\leq 3$ for $r=5$, $m=1$ and $n\geq 6$. Therefore, the $5$-jet of a function $h$ with six or more degrees of freedom belongs to $\sigma_n^5(I)$ at those points $I$ where $h$ is three-jet degenerate. As a consequence, in this case Theorem \ref{nek2} is helpless at establishing whether $h$ is steep or not at those points where it is three-jet degenerate.

\subsection{Genericity and further developments}
As Nekhoroshev pointed out in refs. \cite{Nekhoroshev_1973}, \cite{Nekhoroshev_1979} and as Theorem \ref{nek2} shows, steepness is a generic property in the space of jets of a sufficiently high order $r$, since the codimension of the set containing the jets of all non-steep functions becomes positive for $r$ sufficiently big. Such property is due to the fact that, for increasing $r$, one obtains more and more algebraic conditions that a function must satisfy in order to belong to such set. As Nekhoroshev writes in ref. \cite{Nekhoroshev_1973}: "Hamiltonians that fail to be steep at a non-critical point are infinitely singular: they satisfy an infinite number of conditions on their Taylor coefficients". This, in turn, is a straightforward consequence of Definition \ref{bm}: when $r$ increases, so does the order of the zero that the gradient of the tested function must possess on the minimal path $\gamma$ so to stay in the bad set $\sigma_n^r(I)$. Indeed, since $\gamma$ is a polynomial path, this implies that more and more coefficients of such polynomial must be set equal to zero, which yields an increasing number of algebraic conditions on the coefficients of the jet of the studied function. 
\newline
In the present section, we give some examples of genericity for the sufficient conditions for steepness which we examined throughout the article. 
\newtheorem{ex4}[ex1]{Example}
\begin{ex4}
	Quasi-convexity is a generic property in the space $\mathscr{P}(r,2)$ of polynomials of fixed degree $r\geq 2$ of two variables.  
\end{ex4}
\begin{proof}
In case $h$ is a non quasi-convex polynomial of order two in two variables, there exists $v\neq 0$ such that system 
\begin{equation}\label{2jet}
h^1[v]=0\ ;\ \ h^2[v,v]=0
\end{equation}
 is satisfied. Moreover, $v$ can be normalized to one since the system is homogeneous in such variable. Therefore, for all non quasi-convex functions $h$ and for any integer $r\geq 2$, system (\ref{2jet}) defines an algebraic set of codimension two in the cartesian space $\mathscr{P}(r,2)\times\mathbb{S}^1$ of polynomials and vectors. Since $ \mathbb{S}^1$ has dimension one, by the Theorem of Tarski and Seidenberg, the projection of such algebraic set in the space of polynomials $\mathscr{P}(r,2)$ is semialgebraic and its codimension is no less than $2-1=1$. 
\end{proof}
By following exactly the same startegy, one can prove also the two following 
\newtheorem{ex5}[ex1]{Example}
\begin{ex5}
	Three-jet non-degeneracy is a generic property in the space $\mathscr{P}(r,3)$ of polynomials of fixed degree $r\geq 3$ of three variables.  
\end{ex5}
\newtheorem{ex6}[ex1]{Example}
\begin{ex6}\label{example6}
The sufficient conditions for steepness of Theorem \ref{n4} are generic in the space $\mathscr{P}(r,4)$ of polynomials of fixed degree $r\geq 4$ of four variables.  
\end{ex6}
We remark that the minimal degree $r$ of the polynomials for which genericity holds in the previous examples is the same one yielded by formula (\ref{codimension}) in Theorem \ref{nek2}. Therefore, generic conditions for steepness for polynomials of arbitrary degree can only be inferred if one is able to write sufficient conditions for jets of any order. Such task is not straightforward and will be investigated in future works.
\subsection{On the three-jet non-degeneracy condition}\label{discussion2}
By closely looking at the algebraic form of the sets $\Psi^*_m(n)$ for $n\in\{2,3,4,5\}$ and $m\in\{1,...,n-1\}$, which was developed in the previous sections, one easily sees that any function whose jet belongs to any of these sets must be $3$-jet degenerate. Therefore, if a function depending on a fixed number $n$ of degrees of freedom is three-jet non degenerate, it belongs to the complementary of all sets $\Psi^*_m(n)$, with $m\in\{1,...,n-1\}$. Since for fixed $n\in\{2,3,4,5\}$ the bad set $\sigma_n^5(I)$ is contained in the union of closed sets $\cup_{m\in\{1,...,n-1\}}\Psi^*_m(n)$, by Theorem \ref{nek2} one has that all three-jet non-degenerate functions depending on $n=2,3,4,5$ degrees of freedom are steep. We conjecture that for functions depending on $n\geq 6$ degrees of freedom the same result can be proved by closely looking at the algebraic form of the sets defining the bad set $\sigma_n^r(I)$, for a sufficiently high value of the order $r$. This would constitute an alternative strategy for proving the steepness of three-jet non-degenerate functions with respect to the one contained in \cite{Chierchia_Faraggiana_Guzzo_2019}.

 Finally, by following a similar reasoning as in subsection \ref{ngeq6}, for $r=3$ one obtains $\beta_1\leq 3$ for $n\geq 2$, so that the set of jets of order three satisying the conditions for steepness of Theorem \ref{nek2} is contained in the set of three-jet non degenerate jets. Therefore, three-jet non-degeneracy yields a wider set of steep functions with respect to the construction of Theorem \ref{nek2}. 

\section*{Acknowledgements} 
The present work has been developed autonomously by the author starting from an idea of G. Pinzari on the application of Nekhoroshev's algebraic conditions for steepness to the N-body problem. The author is grateful to G. Pinzari for the mathematical discussion, as well as for providing some unpublished notes containing some ideas about the four-jet condition, which was developed in complete form by Schirinzi and Guzzo in ref. \cite{Schirinzi_Guzzo_2013}. The author also wishes to thank L. Niederman for revising this work and for discussions and new ideas about further developements. 
The author would also like to acknowledge L. Biasco, L. Chierchia, F. Fass\`o, R. Feola, M. Guzzo, S. Mar\`o and J. Massetti for their useful remarks and suggestions.

\bibliographystyle{plainnat}
\bibliography{5-jet-condition-2}

\end{document}